\documentclass[10pt,twoside]{amsart}
\usepackage{latexsym,amssymb,graphicx}
\usepackage{mathrsfs}
\usepackage[all]{xy}
\newtheorem{theorem}{Theorem}[section]
\newtheorem{lemma}[theorem]{Lemma}
\newtheorem{corollary}[theorem]{Corollary}
\newtheorem{proposition}[theorem]{Proposition}

\theoremstyle{definition}
\newtheorem{definition}[theorem]{Definition}

\theoremstyle{remark}
\newtheorem{remark}[theorem]{Remark}
\numberwithin{equation}{section}

\newcommand{\Z}{\mathbb{Z}}
\newcommand{\A}{\mathbb{A}}

\newcommand{\Hom}{\mathrm{Hom}}

\newcommand{\scS}{\mathscr{S}}
\newcommand{\C}{\mathscr{C}}
\renewcommand{\H}{\mathscr{H}}
\newcommand{\ffi}{\varphi}
\newcommand{\colim}{\operatorname{colim}}

\newcommand{\Spec}{\operatorname{Spec}}
\newcommand{\GrO}{{GrO}}
\newcommand{\OGr}{{GrO}}

\newcommand{\Gr}{{Gr}}
\newcommand{\Sch}{\operatorname{Sch}}
\newcommand{\Sm}{\operatorname{Sm}}
\newcommand{\Sets}{\operatorname{Sets}}
\newcommand{\E}{\mathscr{E}}
\newcommand{\F}{\mathscr{F}}
\newcommand{\G}{\mathscr{G}}
\newcommand{\N}{\mathbb{N}}
\newcommand{\scP}{\mathscr{P}}
\renewcommand{\P}{\Bbb{P}}
\newcommand{\Iso}{\operatorname{Iso}}
\newcommand{\St}{\operatorname{St}}
\newcommand{\im}{\operatorname{Im}}
\newcommand{\Cat}{\operatorname{Cat}}
\newcommand{\hocolim}{\operatorname{hocolim}}
\newcommand{\catGrO}{\mathscr{G}rO}

\newcommand{\Qed}{\hfil\penalty-5\hspace*{\fill}$\square$}
\newcommand{\Sing}{\operatorname{Sing_{\bullet}^{\A^1}}}
\newcommand{\eps}{\varepsilon}
\newcommand{\PSh}{\operatorname{PSh}}
\newcommand{\GW}{\mathscr{G}W}

\newcommand{\I}{\mathcal{I}}
\newcommand{\IPS}{\mathcal{S}}
\newcommand{\rk}{\operatorname{rk}}
\newcommand{\catS}{\mathscr{S}}
\newcommand{\scrH}{\mathscr{H}}
\newcommand{\calE}{\mathcal{E}}
\newcommand{\diag}{\operatorname{diag}}
\newcommand{\calH}{\mathcal{H}}

\begin{document}
\bibliographystyle{alpha}

\title[Geometric models for Grothendieck-Witt groups]
{Geometric models for higher Grothendieck-Witt groups in $\A^1$-homotopy theory}

\author{M. Schlichting \and G. S. Tripathi}

\begin{abstract}
We show that the higher Grothendieck-Witt groups, a.k.a. algebraic hermitian $K$-groups, are represented by an infinite orthogonal Grassmannian in the $\mathbb{A}^1$-homotopy category of smooth schemes over a regular base for which $2$ is a unit in the ring of regular functions.
We also give geometric models for various $\P^1$- and $S^1$-loop spaces of hermitian $K$-theory.
\end{abstract}

\date{\today}

\maketitle

\tableofcontents

\section{Introduction}

For a regular noetherian separated scheme $S$ of finite Krull dimension,
denote by ${\H_{\bullet}(S)}$ the pointed unstable $\A^1$-homotopy category of smooth $S$-schemes, and by $[\phantom{12},\phantom{12}]$ or $[\phantom{12},\phantom{12}]_{\H_{\bullet}(S)}$ maps in that category \cite{MorelVoevodsky}.
A theorem of Morel and Voevodsky says that Quillen's algebraic $K$-theory is represented in $\H_{\bullet}(S)$ by $\Z \times BGL \sim \Z \times \Gr_{\bullet}$ 
where for a vector bundle $V$ on $S$, the scheme $Gr_d(V)$ denotes the Grassmannian scheme of $d$-planes in $V$, and $Gr_{\bullet} = \colim_nGr_n(O_S^n\oplus O_S^n)$ is the infinite Grassmannian over $S$.
More precisely \cite[Theorem 3.13, p. 140]{MorelVoevodsky}, for any smooth $S$-scheme $X$ there are natural isomorphisms for all $i\geq 0$ 
\begin{equation}
\label{eqn:MVresult}
K_i(X)\cong [X_+\wedge S^i,\Z\times\Gr_{\bullet}] \cong [X_+\wedge S^i,\Z \times  BGL].
\end{equation}
This is analogous to the fact that complex $K$-theory is represented in topology by $\Z\times BU$ and the infinite complex Grassmannian.

The purpose of this article is to prove a result analogous to (\ref{eqn:MVresult}) for the theory of
higher Grothendieck-Witt groups, a.k.a. algebraic hermitian $K$-groups \cite{karoubi:battelle}, extended to schemes in \cite{myMV}.
Our result has already been used in the work of \cite{AsokFaselCohClass} and \cite{zibrowius} and opens the door to a classification of unstable operations in Grothendieck-Witt theory as done in \cite{riou} for $K$-theory.
To state our main theorem, let $V=(V,\ffi)$ be an inner product space over $S$, that is, a vector bundle $V$ over $S$ equipped with a non-degenerate symmetric bilinear form $\ffi:V\otimes_SV \to O_S$, and let $\OGr_d(V)\subset \Gr_d(V)$ be the open subscheme, of the usual Grassmannian $\Gr_d(V)$ of $d$-planes in $V$, of
those subbundles of $V$ for which the from $\ffi$ restricts to a non-degenerate form.
%For an $S$-scheme $p:X\to S$, the set of $X$-points of $\OGr(V)$ (in the category of $S$-schemes) is the set of sub-bundles $E \subset p^*V$ for which the form $p^*\ffi$ restricts to a non-degenerate form on $E$.
Let $H_S$ be the hyperbolic plane over $S$, 
that is, the rank $2$ vector bundle $O_S^2$ equipped with the inner product $(x,y)\cdot (x',y') = xx'-yy'$.
We define the {\it infinite orthogonal Grassmannian} (over $S$) as the colimit of schemes
$$\OGr_{\bullet}= \colim_n\OGr_{2n}(H^n\perp H^n)$$
where the colimit is taken over the maps 
$$\OGr_{2n}(H^n\perp H^n) \to \OGr_{2n+2}(H^{n+1}\perp H^{n+1}): E \mapsto H\perp E.$$
Moreover, let $O = \colim_n O(H^n)$ be the infinite orthogonal group over $S$ where $O(V)$ denotes the group of isometries of an inner product space $V$.
Let $B_{et}O = \colim_n B_{et}O(H^n)$ be the etale classifying space of $O$ \cite[p. 130]{MorelVoevodsky}.
Finally, for a scheme $X$ with $\frac{1}{2}\in \Gamma(X,O_X)$
let $GW_i(X) = \pi_iGW(X)$ be the $i$-th higher Grothendieck-Witt group of $X$ (\cite[Definition 4.6]{myEx} with $\mathcal{L} = O_X$ and $\eps =1$).
For an affine scheme $X=\Spec A$ (with $\frac{1}{2}\in A$), these groups are Karoubi's hermitian $K$-groups of $A$ \cite[Remark 4.13]{myEx}.
%For two motivic spaces $\X$, $\Y$, the symbol $[\X,\Y]_{\H_{\bullet}(S)}$ denotes the set of morphisms in the pointed unstable $\A^1$-homotopy category $\H_{\bullet}(S)$ of smooth schemes over $S$.
Here is our main result.

\begin{theorem}
\label{thm:GWrep}
Let $S$ be a regular noetherian scheme of finite Krull dimension with $\frac{1}{2}\in \Gamma(S,O_S)$, and let $X$ be a smooth $S$-scheme.
Then there are natural isomorphisms 
$$GW_i(X)\cong [X_+\wedge S^i,\Z\times \OGr_{\bullet}]_{\H_{\bullet}(S)} \cong [X_+\wedge S^i,\Z \times B_{et}O]_{\H_{\bullet}(S)}.$$
\end{theorem}

The proof of the $K$-theory analog of Theorem \ref{thm:GWrep} has two steps.
The first consists in showing that the $K$-theory presheaf $K$ is homotopy invariant and satisfies the Nisnevich Brown-Gersten property.
Both statements follow from Quillen's work \cite{quillen:higherI} and they imply $K_i(X)\cong [X_+\wedge S^i,K]$.
In the second step, one constructs explicit $\A^1$-weak equivalences $\Z\times Gr_{\bullet}\sim_{\A^1}\Z\times BGL \sim_{\A^1} K$.
This was done in \cite{MorelVoevodsky}; see also Remark \ref{CounterexToMorelVoev}.

For higher Grothendieck-Witt theory, the first step was proved by Hornbostel for affine schemes in \cite{hornbostel:A1reps}. 
The extension to non-affine schemes follows from \cite{myMV} 
and is also proved in \cite[Theorems 9.6, 9.8]{myGWDG}.
Thus, $[X_+\wedge S^i,GW]\cong GW_i(X)$. 
Also, it is known from \cite{PaninWalterBO} that $B_{et}O\cong GrO_{\bullet}$ in $\H_{\bullet}(S)$; we give an alternative proof of a more precise version in Proposition \ref{prop:BetGrOdot}.

Denote by $\underline{\Z}$ the constant sheaf $\Z$.
For a ring $R$, denote by $\Delta R$ the standard simplicial ring $n \mapsto \Delta^{n} R=R[T_0,...,T_n]/\langle T_0+\cdots +T_n -1\rangle$.
Theorem \ref{thm:GWrep} is a consequence of the following which is proved in
Theorem \ref{ZBetGsiKG} and Proposition \ref{prop:BetGrOdot} below.

\begin{theorem}
\label{thm:GrO=KOuptoA1}
Let $S$ be a regular noetherian separated scheme of finite Krull dimension with $\frac{1}{2}\in \Gamma(S,O_S)$.
Then there are maps of simplicial presheaves on smooth $S$-schemes
$$\underline{\Z} \times \GrO_{\bullet} \to \underline{\Z} \times B_{et}O \to GW$$
which are weak equivalences of simplicial sets when evaluated at $\Delta R$ 
for any smooth affine $S$-scheme $\Spec R$.
In particular, these maps are isomorphisms in $\scrH_{\bullet}(S)$.
\end{theorem}

We also give models for the $n$-th $\P^1$-loop space of $GW$ and their $S^1$-loop spaces.
Denote by $GW^n(X)$ the $n$-th shifted Grothendieck-Witt space of $X$ (\cite[Definition 7]{myMV} with $\eps = 1$, $Z=X$, $L=O_X$, $\mathcal{A}_X=O_X$, or \cite[Definition 9.1]{myGWDG}), that is, the Grothendieck-Witt space of the category of bounded chain complexes of vector bundles on $X$ with duality in $O_X[n]$, the line bundle $O_X$ placed in degree $-n$.
Let $GW^n:X \mapsto GW^n(X)$ be the corresponding simplicial presheaf made functorial as in \cite[Remark 9.4]{myGWDG}.
Then $GW^0=GW$, and
the presheaves $GW^n$ are homotopy invariant \cite[Theorem 9.8]{myGWDG} and satisfy the Nisnevich Brown-Gersten property \cite[Theorem 9.6]{myGWDG}.
Therefore, 
$$GW^n_i(X)\cong [X_+\wedge S^i,GW^n]$$ 
for all smooth $S$-schemes $X$.
The motivic spaces $GW^n$ are related by $\A^1$-weak equivalences $GW^n \sim \Omega_{\P^1}GW^{n+1}$ (a consequence of the $\P^1$-bundle theorem \cite[Theorem 9.10]{myGWDG}) and isomorphisms
$GW^n \cong GW^{n+4}$ \cite[\S8 Corollary 1]{myMV}, \cite[Remark 5.9]{myGWDG}.
The following is therefore a complete list of geometric models for the $n$-th $\P^1$-loop space $\Omega^n_{\P^1}GW \cong GW^{-n}$ of $\Z\times \OGr_{\bullet}$ and their $S^1$-loop spaces, $n\in \Z$.
Note that upon complex realization we obtain the $8$ spaces of real Bott-periodicity.

\begin{theorem}
\label{thm:8spaces}
Let $S$ be a regular noetherian scheme of finite Krull dimension with $\frac{1}{2}\in \Gamma(S,O_S)$. 
Then there are isomorphisms in $\H_{\bullet}(S)$ 

$$\begin{array}{ll}
GW^n \cong \left\{ \begin{array}{ll}
\Z \times \OGr_{\bullet} & n=0 \\
Sp/GL & n=1 \\
\Z \times BSp & n=2 \\
O/GL & n=3 
\end{array}
\right.
&\hspace{4ex}
\Omega_{S^1}GW^n \cong \left\{ \begin{array}{ll}
O & n=0 \\
(GL/O)_{et} & n=1 \\
Sp & n=2 \\
GL/Sp & n=3 
\end{array}
\right.
\end{array}$$
where $Sp$ denotes the infinite symplectic group and $(GL/O)_{et}$ denotes the etale or scheme theoretic quotient.
\end{theorem}

More precise versions are proved in Theorems \ref{ZBetGsiKG} and \ref{thm:Modles}.
\vspace{2ex}

\noindent
{\bf Acknowledgments}.
The authors would like to thank Paulo Lima-Filho and Simon Markett for useful discussions.
We would also like to acknowledge support through NSF grant DMS 0906290.

\section{Orthogonal Grassmannians}
\label{GrO}

For a quasi-compact, separated scheme $S$, denote by
 $\Sch_S$ and $\Sm_S$ the categories of separated, finite type $S$-schemes and its full subcategory of smooth $S$-schemes, respectively.
%\edit{$S$ quasi-compact and separated? infinite sums of shvs}

Let $\F$ be a quasi-coherent sheaf on a scheme $X$.
A {\em symmetric bilinear form} on $\F$ is a map 
$\ffi: \F\otimes_X\F \to O_X$ of $O_X$-modules such that $\ffi\tau =\ffi$ where $\tau:\F\otimes \G \cong \G\otimes \F$ is the switch map.
The form $\ffi$ is called {\em non-degenerate} and the pair $(\F,\ffi)$ is called an {\em inner product space} if $\F$ is a vector bundle and the adjoint
$\hat{\ffi}: \F \to \F^*= Hom_{O_X}(\F,O_X): \xi\mapsto \ffi(\hspace{2ex}\otimes \xi)$ is an isomorphism.
If $g:\G \to \F$ is a map of $O_X$-modules, then the restriction $\ffi_{|\G}$ of $\ffi$ to $\G$ has as adjoint the map $g^*\hat{\ffi} g$.
%In particular, if $(\G, \ffi_{|\G})$ is non-degenerate, then $\G$ is a direct factor of $\F$, since the map $g: \G \to \F$ has left inverse $\ffi_{|\G}^{-1}g^*\ffi$.
If $\F$ is a sheaf on $S$ and $p:X\to S$ is an $S$-scheme, we may write $\F_X$ for the sheaf $p^*\F$.

\begin{definition}[Orthogonal Grassmannians]
\label{dfn:GrO}
Let $\F=(\F,\ffi)$ be a quasi-coherent sheaf over $S$ together with a symmetric bilinear form $\ffi: \F\otimes_S\F \to O_S$ which may be degenerate.
The {\em Grassmannian of non-degenerate subspaces of $\F$} is the presheaf
$$\OGr(\F):(\Sch_S)^{op}\to \Sets$$
whose value at an $S$-scheme $p:X\to S$ is the set 
$\GrO(\F_X)$ of finite rank locally free $O_X$-submodules $E\subset \F_X$ of $\F_X = p^*\F$ for which the restriction $\ffi_{|E}$ of the form $\ffi$ to $E$ is non-degenerate.
For a map $f:X\to Y$ of $S$-schemes, the map $\GrO(\F_Y)\to \GrO(\F_X)$ is induced by the pullback $f^*$ of quasi-coherent sheaves.
For an integer $d\geq 0$ we let
$$\GrO_d(\F) \subset \GrO(\F)$$
be the subpresheaf of those non-degenerate subspaces $E\subset \F$ which have constant rank $d$.
If $X=\Spec R$ is affine, we may write $\GrO_d(\F_R)$ and $\GrO(\F_R)$ in place of $\GrO_d(\F_X)$ and $\GrO(\F_X)$.
\end{definition}

\begin{lemma}
\label{lem:GrOsmoothaffine}
Let $V = (V,\ffi)$ be an inner product space of rank $n$ over $S$, and $0\leq d \leq n$ be an integer.
Then the presheaf $\GrO_d(V)$ is represented by a scheme which is smooth and affine over $S$.
\end{lemma}

\begin{proof}
To see that $\GrO_d(V) \to S$ is smooth, we note that it is an open subscheme of the usual Grassmannian $\Gr_d(V)$ of $d$-planes in $V$.
More precisely, if we denote by $\xi$ the universal rank $d$ subbundle of $V$ on $\Gr_d(V)$, then the form on $V$ restricts to a (degenerate) form $\ffi_{|\xi}$ on $\xi$, and 
$\GrO_d(V)$ is the open subscheme of $\Gr_d(V)$ where $\ffi_{|\xi}$ is non-degenerate, that is, $\GrO_d(V)$ is the non-vanishing locus of the global section $\Lambda^d\hat{\ffi}$ of the line bundle $\underline{\Hom}_{O_X}(\Lambda^d\xi,\Lambda^d\xi^*)$ on $X=\Gr_d(V)$.
Since $\Gr_d(V)\to S$ is smooth, so is $\GrO_d(V) \to S$.

To see that $\GrO_d(V)\to S$ is an affine morphism, note that for any $S$-scheme $X$, we have a natural bijection of sets
$$\GrO(V_X) \cong \{ p \in \Hom_{O_X}(V_X,V_X)|\ p=p^2,\  p^*\ffi = \ffi p\}.$$
The map is defined by $(i:M \subset V_X) \mapsto i(\ffi_{|M}^{-1})i^*\ffi$
and has inverse $p\mapsto \im(p)\subset V_X$.
This shows that $\GrO(V)$ is a closed subscheme of the vector bundle  $\underline{\Hom}_{O_S}(V,V)$ over $S$ defined by two equations.
In particular, $\GrO(V) \to \underline{\Hom}_{O_S}(V,V) \to S$ are affine morphisms.
As a closed subscheme of $\GrO(V)$, the scheme $\GrO_d(V)$ is also affine over $S$.
\end{proof}

For an $S$-scheme $X$, let $H_X$ be the hyperbolic plane over $X$, that is, the rank $2$ vector bundle $O_X^2$ equipped with the inner product $(x,y)\cdot (x',y')=xy-x'y'$.
Let
$H^n_X$ be its $n$-fold orthogonal sum (an inner product space over $X$) and let $H^{\infty}_X = \colim_nH^n_X$ be the infinite hyperbolic space (a quasi-coherent $O_X$-module with symmetric bilinear form).
%We may write $H^n$ and $H^{\infty}$ in place of $H^n_S$ and $H^{\infty}_S$.
%For an affine scheme $X=\Spec R$, we may write $\GrO(H^n_R)$ and $\GrO(H^\infty_R)$ instead of $\GrO(H^n_X)$ and $\GrO(H^\infty_X)$.
%Note that these sets are independent of the base-scheme $S$.
Order non-degenerate subspaces of $H^{\infty}_X$ by inclusion.
This defines a filtered category $\calH$.
Its objects are non-degenerate subspaces $V\subset H^{\infty}$ (which are inner product spaces), and maps are inclusions of subspaces.
For a non-degenerate subspace $V \subset V'$ of an inner product space $V'$, denote by $V'-V$ the orthogonal complement of $V$ in $V'$.

\begin{definition}[Infinite orthogonal Grassmannian]
\label{dfn:infiniteGrO}
For a vector bundle $V$ of constant rank, write $|V|$ for its rank. 
The {\em infinite orthogonal Grassmannian} over $S$ is the presheaf
$$GrO_{\bullet}= \colim_{V\subset H^{\infty}_S} GrO_{|V|}(V\perp H^{\infty}).$$
The colimit is taken over the non-degenerate subbundles of $H^{\infty}_S$ of constant rank ordered by inclusion, and  the transition maps are   
$$GrO_{|V|}(V\perp H^{\infty}) \to GrO_{|V'|}(V'\perp H^{\infty}): E \mapsto (V'-V)\perp E$$
whenever $V\subset V'$.
Of course, it suffices to take the colimit over a cofinal subset such as the set $\{H^n_S|\ n\in \N\}$.
\end{definition}

\section{The etale classifying space}

Let $S$ be a scheme and 
$\F=(\F,\ffi)$ a quasi-coherent sheaf over $S$ together with a symmetric bilinear form $\ffi: \F\otimes_S\F \to O_S$ which may be degenerate.
For an $S$-scheme $X$, denote by 
$$\IPS(\F_X)$$
the category of inner product spaces embedded in $\F_X$, that is, the category
whose objects are the locally free $O_X$-submodules $E\subset \F_X$ of $\F_X = p^*\F$ for which the restriction $\ffi_{|E}$ of the form $\ffi$ to $E$ is non-degenerate.
A map from $E_0\subset \F_X$ to $E_1\subset \F_X$ is an isometry 
$(E_0,\ffi_{|E_0}) \to (E_1,\ffi_{|E_1})$ which does not need to be compatible with the embeddings $E_0,E_1\subset \F_X$.
For a map $f:X \to Y$ of $S$-schemes, pull-back $f^*$ of quasi-coherent modules defines a map $\IPS(\F_Y) \to \IPS(\F_X)$, and we obtain a presheaf of categories
$X \mapsto \IPS(\F_X)$.
Note that the set of objects of $\IPS(\F_X)$ is precisely $\GrO(\F_X)$.

For an integer $d\geq 0$, we denote by
$$\IPS_d(\F_X) \subset \IPS(\F_X)$$
the full subcategory of those inner product spaces $E \subset \F_X$ which have constant rank $d$.
Then $\IPS_d(\F)$ is a presheaf of groupoids
%\edit{in fact sheaf?}
with presheaf of objects $\GrO_d(\F)$.

In the category $\IPS_{|V|}(V\perp H^{\infty})$,
the group of automorphisms of the object $V\subset V\perp H^{\infty}:v\mapsto (v,0)$ is the group $O(V)$ of isometries of $V$.
Thus we have a full inclusion $O(V) \to \IPS_{|V|}(V\perp H^{\infty})$ of presheaves of categories.
After etale sheafification, this inclusion becomes an equivalence of categories.
This is because in a strictly henselian ring $R$ with $\frac{1}{2}\in R$, every unit is a square, and thus, any two innner product spaces over $R$ 
are isometric if and only if they have the same rank.
It follows that the inclusion of categories induces a map of simplicial presheaves $BO(V) \to B\IPS_{|V|}(V\perp H^{\infty})$ which is a weak equivalence at all strictly henselian $R$ with $\frac{1}{2}\in R$.
In other words,
this map is a weak equivalence in the etale topology.
In particluar, a globally fibrant model of $B \IPS_{|V|}(V\perp H^{\infty})$ for the etale topology is also a globally fibrant model, denoted $B_{et}O(V)$, of $BO(V)$. 
Therefore, we obtain a sequence of maps 
\begin{equation}
\label{eqn:IPStoBetO}
B O(V) \to B \IPS_{|V|}(V\perp H^{\infty}) \to B_{et}O(V)
\end{equation}
which are weak equivalences in the etale topology, and the last presheaf is fibrant (in the etale topology).

\begin{lemma}
\label{lem:IPSisBetO}
Let $V$ be an inner product space over a scheme $S$ with $\frac{1}{2}\in \Gamma(S,O_S)$.
Then for any affine $S$-scheme $\Spec R$, the map
$$B \IPS_{|V|}(V\perp H^{\infty})(R) \to B_{et}O(V)(R)$$
is a weak equivalence of simplicial sets.
In particular, the following map is an $\A^1$-weak equivalence
$$B \IPS_{|V|}(V\perp H^{\infty}) \to B_{et}O(V).$$
\end{lemma}

\begin{proof}
This follows from \cite{JardineStacks}.
Let $St$ be the stack associated with the sheaf of groupoids $\IPS_{|V|}(V\perp H^{\infty})$, then $St$ is a sheaf of groupoids satisfying the effective descent condition for the etale topology.
So, $X\mapsto St(X)$ is a sheaf version of the category of $O(V)$-torsors over $X$.
Since for affine $X$, the category $\IPS_{|V|}(V\perp H^{\infty}_X)$ is already the category of all $O(V)$-torsors, the map 
$\IPS_{|V|}(V\perp H^{\infty})(X) \to St(X)$ is an equivalence of categories for $X$ affine.
Therefore, in the string of maps
$$B \IPS_{|V|}(V\perp H^{\infty})(X) \to B St(X) \to B_{et}O(V)(X),$$
the first map is a weak equivalence for every affine $S$-scheme $X$.
The second map $B St(X) \to B_{et}O(V)(X)$ is a weak equivalence of simplicial sets for all $S$-schemes $X$ \cite[Theorem 6]{JardineStacks}.
\end{proof}

\begin{definition}
\label{dfn:Sbullet}
Set
$$\IPS_{\bullet} = \colim_{V\subset H^{\infty}_S} \IPS_{|V|}(V\perp H^{\infty})$$
where for $V \subset V'$, the transition map 
$\IPS_{|V|}(V\perp H^{\infty}) \to \IPS_{|V'|}(V'\perp H^{\infty})$ is defined by $E\mapsto (V'-V)\perp E$ on objects and by $g\mapsto 1_{V'-V}\perp g$ on maps.
\end{definition}

Inclusion of zero-simplices and the second map in (\ref{eqn:IPStoBetO})
define the string of maps of simplicial presheaves
$$\GrO_{|V|}(V\perp H^{\infty}) \to \IPS_{|V|}(V\perp H^{\infty}) \to B_{et}O(V)$$
in which the second map is section-wise a weak equivalence on affine schemes, by Lemma \ref{lem:IPSisBetO}.
Passing to the colimit over the index category $\calH$ defines the string of maps
\begin{equation}
\label{eqn:GrOIPSBetOstable}
\GrO_{\bullet} \to \IPS_{\bullet} \to B_{et}O
\end{equation}
in which the second map is a weak equivalence when evaluated at affine schemes.
%We will show in Proposition \ref that this map is a weak equivalence of simplicial sets when evaluated at the standard simplicial ring $\Delta R$ for any ring $R$ with $\frac{1}{2}\in R$.

\section{The Grothendieck-Witt space}

Let $R$ be a commutative ring.
Let $\scS_R$ denote the category of inner product spaces over 
$R$ with isometries as morphisms.
This category is symmetric monoidal with respect to orthogonal sum $\perp$.
In particular, we have the category $\scS_R^{-1}\scS_R$ as constructed in \cite{quillenGrayson}
whose classifying space $B\scS_R^{-1}\scS_R$ is naturally weakly equivalent to $GW(R)$ \cite{mygiffen}, \cite[Appendix A]{myGWDG} (at least when $\frac{1}{2}\in R$ though this is also true without this hypothesis).
Recall that the objects of $\scS_R^{-1}\scS_R$  are pairs of inner product spaces and a map $(A_0,A_1) \to (B_0,B_1)$ in that category is an equivalence class of data $[C,a_0,a_1]$ where $C$ is an inner product space and $a_i: A_i\perp C \to B_i$ is an isometry for $i=0,1$.
We have $[C,a_0,a_1]=[C',a'_0,a'_1]$ if and only if there is an isometry $f: C \cong C'$ such that $a'_i(1_{A_i}\perp f) = a_i$ for $i=0,1$.

The category $\scS_R^{-1}\scS_R$ is not convenient for our purposes as
it is, a priori, not a small category,  and it is not really functorial in $R$.
In particular, the assignment $X \mapsto \scS_R^{-1}\scS_R$ with $R=\Gamma(X,O_X)$ does not define a presheaf.
We remedy this as follows.

\begin{definition}[The presheaf of Grothendieck-Witt spaces]
Let
$$\GW(R) \subset \scS_R^{-1}\scS_R$$
be the full subcategory 
whose objects are pairs
$(A,B)$ where $A \subset H^{\infty}_R\perp H^{\infty}_R$ and $B \subset  (H^{\infty})_R^{\perp 3}$ are finitely generated non-degenerate subspaces of $(H^{\infty}_R)^{\perp 2}$ and $(H^{\infty}_R)^{\perp 3}$, respectively.
The ambient bilinear form spaces $(H^{\infty}_R)^{\perp 2}$ and $(H^{\infty}_R)^{\perp 3}$ are chosen so that we can construct certain maps below.
The explicit ambient spaces don't matter as long as they are functorial in $R$ and contain a copy of each inner product space over $R$.

From our definition, the category $\GW(R)$ is small, it is equivalent to $\scS_R^{-1}\scS_R$, and it is functorial in $R$.
In particular, the assignment $X \mapsto \GW(R)$ with $R=\Gamma(X,O_X)$ does define a presheaf (of categories and hence of simplicial sets after application of the nerve functor). 
\end{definition}

From \cite{mygiffen}, \cite[Appendix A]{myGWDG}, there is a map of presheaves
$\GW \to GW$ which is a weak equivalence (of simplicial sets) for all affine schemes.
We record a special case in the following Lemma.

\begin{lemma}
Let $S$ be a regular separated noetherian scheme of finite Krull dimension with $\frac{1}{2}\in S$.
Then map of presheaves 
$\GW \to GW$ in $\Delta^{op}\PSh(\Sm_S)$ is a weak equivalence of simplicial sets at all $\Spec R \to S$.
In particular the map of presheaves is a Nisnevich simplicial weak equivalence, and hence an $\A^1$-weak equivalence.
\Qed
\end{lemma}

\begin{definition}
We define the presheaf of {\em reduced Grothendieck-Witt spaces} $\widetilde{\GW}$ 
as the presheaf of categories which for a ring $R=\Gamma (X,O_X)$ is the full subcategory 
$$\widetilde{\GW}(R) \subset \GW(R)$$ 
of objects $(A,B) \in \GW(R)$ such that 
$A \subset H^{\infty}_R =0\perp H^{\infty}_R \subset (H^{\infty}_R)^{ \perp 2}$, and $B \subset A \perp H^{\infty}_R \subset 0 \perp (H^{\infty}_R)^{\perp 2} \subset (H^{\infty}_R)^{ \perp 3}$ and
$A,B$ have the same constant rank.
For an integer $i$, we set $\widetilde{GW}_i(R) = \pi_i(\widetilde{\GW}(R))$ where the homotopy groups are taken with respect to the base point $(0,0)$.
\end{definition}

Consider the integers $\Z$ as a (symmetric monoidal) category with one object for each integer and only identity morphisms.
Let $\N \subset\Z$ be the (full) subcategory of non-negative integers viewed as a symmetric monoidal category where the monoidal product is given by addition.
 
The functor 
\begin{equation}
\label{eqn:NNisZ}
\N^{-1}\N \to \Z: (n,m) \mapsto n-m
\end{equation}
%\edit{check that this is indeed a functor}
induces a weak equivalence of simplicial sets (after application of the nerve functor) since all fibres are filtered categories and hence contractible.

Consider the ring $R$ as an inner product space with bilinear form $R \otimes R \to R: x\otimes y \mapsto xy$.
Then we have a map of presheaves of categories 
\begin{equation}
\label{eqn:NNtoGW}
\N^{-1}\N \to \GW: (n,m)\mapsto (R^n,R^m)
\end{equation}
where the first factor $R^n$ is considered as being in $H^n\perp 0 \subset H^{\infty} \perp H^{\infty}$ and the second factor $R^m$ as being in 
$H^m \perp 0 \perp 0 \subset H^{\infty}\perp H^{\infty} \perp H^{\infty}$.
Together with the inclusion $\widetilde{\GW} \subset \GW$ this defines a 
 map of presheaves of categories
\begin{equation}
\label{eqn:ZtimesTildeGWisGW}
\N^{-1}\N \times \widetilde{\GW} \to \GW: (n,m), (A,B) \mapsto (R^n \perp A,R^m\perp B)
\end{equation}

\begin{lemma}
Let $R$ be a connected ring with $\frac{1}{2}\in R$.
Then the map (\ref{eqn:ZtimesTildeGWisGW}) is a weak equivalence of simplicial sets.
In particular, the maps (\ref{eqn:ZtimesTildeGWisGW}) and (\ref{eqn:NNisZ}) induce $\A^1$-weak equivalences
$$\Z \times \widetilde{\GW} \stackrel{\sim}{\leftarrow}  \N^{-1}\N \times \widetilde{\GW} \stackrel{\sim}{\to} \GW$$
\end{lemma}

\begin{proof}
For a connected ring $R$, the functor of categories 
$\GW(R) \to \Z: (A,B) \mapsto \rk A - \rk B$ is well defined, and the functor 
(\ref{eqn:NNtoGW}) provides a splitting.
\end{proof}

\begin{remark}[$K$-theory as a homotopy colimit]
\label{rem:SSasHocolim}
Let $\I$ be the category whose objects are the finitely generated non-degenerated subspaces $V \subset H^{\infty}$ and whose maps are all isometric embeddings, that is, a map from $V\subset H^{\infty}$ to $W\subset H^{\infty}$ is a map of $O_X$-modules $f:V \to W$ such that the form on $W$ restricts to the form on $V$ but $f$ does not need to commute with the embeddings $V,W \subset H^{\infty}$.
Composition is composition of $O_X$-module maps.
In the notation of \cite{quillenGrayson}, the category $\I$ is the category $\langle \IPS({H^{\infty}}),\IPS({H^{\infty}})\rangle$.
Note that our index category $\calH$ is naturally a subcategory of $\I$.
It is the subcategory with the same objects as $\I$ and with maps those isometric embeddings $f:V\to W$ which do commute with the embedding $V,W\subset H^{\infty}$.

We define a functor 
from $\I$ to the category of small categories which on objects is 
$$V \mapsto \IPS_{|V|}(V \perp H^{\infty})$$ 
and which sends an isometric embedding $g:V \to W$ to the functor
$$
\renewcommand\arraystretch{1.5}
\begin{array}{ll}
\IPS_{|V|}({V \perp H^{\infty}}) \to \IPS_{|W|}({W \perp H^{\infty}}): & E \mapsto (W-g(V))^-\perp g(E)\\
&  e \mapsto 1_{(W-g(V))^-} \perp  geg^{-1}.
\end{array}$$
Then we have an equality of categories
$$\widetilde{\GW}(R) = \hocolim_{V \in \I}\IPS_{|V|}({V\perp H^{\infty}})$$
where the right hand side is the homotopy colimit of categories as in \cite{Thomason:hocolim} whose construction is recalled in Appendix \ref{dfn:hocolim}.

Replacing $\IPS_{|V|}(V \perp H^{\infty})$ with the full groupoid $\IPS(V \perp H^{\infty})$ of all inner product spaces in $V \perp H^{\infty}$ and taking the homotopy colimit as above yields a model for the Grothendieck-Witt space $GW(R)$ of $R$.
\end{remark}

\section{The maps $\GrO_{\bullet} \to B_{et}O \to \widetilde{\GW}$}
\label{sec:TheMaps}

\begin{definition}
\label{dfn:GrOtoGW}
By Remark \ref{rem:SSasHocolim} that the (reduced) Grothendieck-Witt space is a homotopy colimit.
Therefore, we will need to express the presheaves $\GrO_{\bullet}$ and $\IPS_{\bullet} \simeq B_{et}O$ as homotopy colimits as well.
Recall that the presheaves $\GrO_{\bullet}$ and $\IPS_{\bullet}$ are obtained as the colimits of the simplicial sets $\GrO_{|V|}(V\perp H^{\infty})$ and $\IPS_{|V|}(V\perp H^{\infty})$
over the index category $\calH$ of non-degenerate subspaces $V\subset H^{\infty}$.
Replacing the colimit by their homotopy colimits as in Appendix \ref{dfn:hocolim} yields the definition of the 
presheaves of categories $\catGrO_{\bullet}$ and $\catS_{\bullet}$ which for $R = \Gamma(X,O_X)$ are 
$$
\renewcommand\arraystretch{2}
\begin{array}{lll}
\catGrO_{\bullet}(R)& =& \hocolim_{V\subset H^{\infty}} \GrO_{|V|}(V_R\perp H^{\infty}_R)\\
\catS_{\bullet}(R) &= & \hocolim_{V\subset H^{\infty}} \IPS_{|V|}(V_R\perp H^{\infty}_R)
\end{array}$$
together with weak equivalences of presheaves of simplicial sets
\begin{equation}
\label{eqn:CatGroToCatS}
\catGrO_{\bullet} \stackrel{\sim}{\to} \GrO_{\bullet},\hspace{3ex}\catS_{\bullet} \stackrel{\sim}{\to} \IPS_{\bullet};
\end{equation}
see Lemma \ref{lem:FilteringHocolim=colim}.

The natural transformation of functors $\calH \to \Cat$
which at $V\in \calH$ is the inclusion of zero-simplices $\GrO_{|V|}(V\perp H^{\infty}) \to \IPS_{|V|}(V\perp H^{\infty})$ defines a map of presheaves of categories
$$\catGrO_{\bullet} \to \catS_{\bullet}$$
Furthermore, the inclusion 
$\calH \subset \I$ 
defines a functor
$$\hocolim_{V\in \calH}S_{|V|}(V\perp H^{\infty}) \longrightarrow 
\hocolim_{V\in \I}S_{|V|}(V\perp H^{\infty}),$$
that is, a map of presheaves of categories 
\begin{equation}
\label{eqn:CatSToGW}
\catS_{\bullet} \to \widetilde{\GW}.
\end{equation}
\end{definition}

Write $\Delta R$ for the simplicial ring with $n\mapsto \Delta^nR = R[T_0,...,T_n]/(T_0 + \cdots + T_n - 1)$.

\begin{theorem}
\label{thm:simplHtpyEq}
Let $R$ be a commutative connected regular noetherian ring with $\frac{1}{2}\in R$.
Then the maps (\ref{eqn:CatGroToCatS}) and (\ref{eqn:CatSToGW})
induce weak equivalences of simplicial sets
$$\catGrO_{\bullet}(\Delta R) \stackrel{\sim}{\rightarrow} \catS_{\bullet}(\Delta R) \stackrel{\sim}{\rightarrow} \widetilde{\GW}(\Delta R).$$
\end{theorem}

The proof is in Corollary \ref{cor:GrOisS} and Proposition \ref{prop:SGWHweak} in view of the weak equivalences (\ref{eqn:CatGroToCatS}).

\section{Setting up the proof of Theorem \ref{thm:simplHtpyEq}}
\label{sec:H*Proof}

Let $R$ be a commutative ring, $V$ an inner product space of constant rank over $R$, and $U$ an $R$-module equipped with a symmetric bilinear form.
Denote by 
$$\GrO_V(U) \subset \GrO_{|V|}(U)$$
 the subset of those non-degenerate subspaces $W \subset U$ which are isometric to $V$.
Scalar extension makes $\GrO_V(U)$ into a presheaf on affine $R$-schemes.
Similarly, denote by
$$\IPS_V(U) \subset \IPS_{|V|}(U)$$
the presheaf of full subcategories of those non-degenerate subspaces $W \subset U$ which are isometric to $V$.
The presheaf of objects of $\IPS_V(U)$ is $\GrO_V(U)$.

Let 
$\Iso_d(R)$
denote the set of isometry classes of inner product spaces over $R$ of constant rank $d$.
We define a map of sets
$$\GrO_d(V\perp H^{\infty}_R) \to \Iso_d(R): E\mapsto [E]$$ by sending a finitely generated non-degenerate subspace $E$ of  $V \perp H^{\infty}_R$ to its isometry class $[E] \in Iso_d(R)$.
Similarly, we define a map of categories
$$\IPS_d(V\perp H^{\infty}_R) \to \Iso_d(R): E\mapsto [E].$$
For an inner product space $V$ over $R$ of constant rank $d$, denote by $V:* \to \Iso_d(R)$ the map sending the point $*$ to the class $[V]$ of $V$ in $\Iso_d(R)$.
By definition, we have a cartesian diagram of sets
\begin{equation}
\label{eqn:cartdiag1}
\xymatrix{
\GrO_V(V\perp H^{\infty}_R) \ar[r] \ar[d] & \GrO_{|V|}(V\perp H^{\infty}_R)\ar[d]\\
{*}\ar[r]^V & \Iso_{|V|}(R)
}
\end{equation}
and of categories
\begin{equation}
\label{eqn:cartdiagS1}
\xymatrix{
\IPS_V(V\perp H^{\infty}_R) \ar[r] \ar[d] & \IPS_{|V|}(V\perp H^{\infty}_R)\ar[d]\\
{*}\ar[r]^V & \Iso_{|V|}(R).
}
\end{equation}
Taking the colimit over the non-degenerate subspaces $V \subset H^{\infty}_R$ with transition maps as in Definitions \ref{dfn:infiniteGrO} and \ref{dfn:Sbullet},
we obtain the cartesian diagrams of simplicial sets
$$\xymatrix{
\GrO_{[0]}(R) \ar[r] \ar[d] & \GrO_{\bullet}(R)\ar[d] &\IPS_{[0]}(R) \ar[r] \ar[d] & \IPS_{\bullet}(R)\ar[d] \\
{*}\ar[r] & \widetilde{GW}_0(R) & {*}\ar[r] & \widetilde{GW}_0(R)
}$$
where the upper left corners are 
$\GrO_{[0]}(R) = \colim_{V\subset H^{\infty}_R}\GrO_V(V\perp H^{\infty}_{R})$ and
%\hspace{3ex}\text{and}\hspace{3ex}
$\IPS_{[0]}(R) = \colim_{V\subset H^{\infty}_R}\IPS_V(V\perp H^{\infty}_{R})$.
% and the lower right corner $\widetilde{GW}_0(R)$ is the Grothendieck-Witt group inner product spaces over $R$.

\begin{lemma}
\label{lem:GrOfibr}
Let $R$ be a connected regular ring with $\frac{1}{2} \in R$.
Then the cartesian diagrams of simplicial sets
$$\xymatrix{
\GrO_{[0]}(\Delta R) \ar[r] \ar[d] & \GrO_{\bullet}(\Delta R)\ar[d] &\IPS_{[0]}(\Delta R) \ar[r] \ar[d] & \IPS_{\bullet}(\Delta R)\ar[d] \\
{*}\ar[r] & \widetilde{GW}_0(\Delta R) & {*}\ar[r] & \widetilde{GW}_0(\Delta R)
}$$
are homotopy cartesian, and the lower right corners are constant simplicial sets.
\end{lemma}

\begin{proof}
The Grothendieck-Witt group $GW_0$ is homotopy invariant for regular rings (with $2$ a unit).
For connected rings, the kernel $\tilde{GW}_0$ of the rank map $GW_0 \to \Z$ is therefore also homotopy invariant.
It follows that the lower right corner of the diagram is a constant simplicial set.
Hence,  the lower horizontal map is a fibration of (constant) simplicial sets.
\end{proof}

Diagram (\ref{eqn:cartdiag1}) maps to diagram (\ref{eqn:cartdiagS1}) via the inclusion of zero simplices.
By Lemma \ref{lem:GrOfibr}, we have a map of homotopy fibrations 
\begin{equation}
\label{eqn:GrOKOdiagram}
\xymatrix{
\GrO_{[0]}(\Delta R) \ar[r] \ar[d] & \GrO_{\bullet}(\Delta R) \ar[r] \ar[d]&\widetilde{GW}_0(\Delta R) \ar[d]^1\\
\IPS_{[0]}(\Delta R) \ar[r] & \IPS_{\bullet}(\Delta R) \ar[r] &\widetilde{GW}_0(\Delta R).
}
\end{equation}

The rest of this section is devoted to the proof of the following.

\begin{proposition}
\label{prop:GrO0=KO0}
Let $R$ be a commutative ring with $\frac{1}{2} \in R$
and $V$ an inner product space over $R$.
Then we have weak equivalences of simplicial sets
$$\GrO_V(V\perp H^{\infty}_{\Delta R}) \stackrel{\sim}{\to} \IPS_V(V\perp H^{\infty}_{\Delta R}) \stackrel{\sim}{\leftarrow} BO(V_{\Delta R}),$$
where the first map is inclusion of zero-simplices and the second map is
the inclusion of the endomorphism category of the object $V$ into $\IPS_V(V\perp H^{\infty})$.
In particular, we have 
weak equivalences of simplicial sets
$$\GrO_{[0]}(\Delta R) \stackrel{\sim}{\to} \IPS_{[0]}(\Delta R) \stackrel{\sim}{\leftarrow} BO({\Delta R}).$$
\end{proposition}

Let 
$$O(V\perp H^{\infty}) = \colim_{W\subset V\perp H^{\infty}}O(W)$$
be the infinite orthogonal group based on $V\perp H^{\infty}$.
It is the filtered colimit over the poset of finitely generated non-degenerate subspaces $W$ of $V\perp H^{\infty}$ of the isometry groups $O(W)$ of $W$ where
for an inclusion $W \subset W'$, we embed $O(W)$ into $O(W')$ via $a \mapsto a \perp id_{W'-W}$.
Our next aim is to identify the simplicial set 
$\GrO_V(V\perp H^{\infty}_{\Delta R})$ with the simplicial set $BO(V_{\Delta R})$, up to homotopy.
We will need the following lemma.

\begin{lemma}
\label{lem:OH=OVH}
Let $V$ be an inner product space over a commutative ring $R$ with $\frac{1}{2}\in R$.
Then the inclusion $H^{\infty} \subset V\perp H^{\infty}$ induces a homotopy equivalence of simplicial groups
$$O(H^{\infty}_{\Delta R}) \to O(V\perp H^{\infty}_{\Delta_R}):A \mapsto 1_V\perp A.$$
\end{lemma}

\begin{proof}
We first prove the claim when $V=H$.
So, we need to show that $j:O(H^{\infty}_{\Delta_R}) \to O(H^{\infty}_{\Delta_R}): A \mapsto 1_H\perp A$ is a homotopy equivalence.
The point is that the two inclusions $j: O(H^n) \to O(H^{2n+2}): A \mapsto 1_H\perp A \perp 1_{H^{n+1}}$ and 
$i: O(H^n) \to O(H^{2n+2}): A \mapsto A \perp 1_{H^{n+2}}$
are naively $\A^1$-homotopic (see Appendix \ref{dfn:naiveA1htpy} for a definition).
This is because $i = g \cdot j \cdot g^{-1}$ where 
$g= H(h \oplus h^{-1})$, $H: GL_{n+1}(R) \to O(H^{n+1})$ is the hyperbolic map and $h = \left( \begin{smallmatrix} 0 & 1 \\ I_n & 0\end{smallmatrix}\right) \in GL_{n+1}(R)$ with $I_n\in GL_n(R)$ the identity matrix.
Now, $h\oplus h^{-1} \in GL_{2n+2}(R)$ is a product of elementary matrices each of which is naively $\A^1$-homotopic to the identity by an elementary $\A^1$-homotopy.
Therefore, $g$ is naively $\A^1$-homotopic to the identity and the inclusions 
$j = g\cdot i \cdot g^{-1}: O(H^n_{\Delta R}) \to O(H^{2n+2}_{\Delta R})$ and 
$i: O(H^n_{\Delta R}) \to O(H^{2n+2}_{\Delta R})$ are simplicially homotopic via a base-point preserving homotopy, by Lemma \ref{lem:naiveA1htpy}.
It follows that $j: \pi_kO(H^{\infty}_{\Delta R}) = \colim_n \pi_kO(H^{n}_{\Delta R}) \to \pi_kO(H^{\infty}_{\Delta R})$ is the identity map, hence an isomorphism for all $k\geq 0$.
Since $O(H^{\infty}_{\Delta R})$ is an $H$-group, this implies the claim for $V=H$.

By induction, the claim is true for $V=H^n$.
For general $V$, choose an embedding $V \subset H^n$.
Then the composition of the first two and the composition of the last two maps in the following diagram are homotopy equivalences 
$$
O(H^{\infty}_{\Delta R}) \to O(V\perp H^{\infty}_{\Delta_R}) \to O(H^n \perp H^{\infty}_{\Delta R}) \to O(H^n \perp V \perp H^{\infty}_{\Delta_R})$$
since $V\perp H^{\infty} \cong H^{\infty}$.
This finishes the proof of the Lemma.
\end{proof}

Let $V = (V,\phi_V)$ be an inner product space over a commutative ring $R$, and  let $U = (U,\phi_U)$ be an $R$-module equipped with a symmetric bilinear form.
For a commutative $R$-algebra $A$,  let
$$\St(V,U)(A)$$
be the set of isometric embeddings $f:V_A \to U_A$ over $A$, that is, the set of those $A$-linear maps $f:V_A \to U_A$ such that $\phi_V=f^*\phi_Uf$.
Then $\St(V,U)$ is a presheaf on affine $R$-schemes.

The group $O(V\perp H^{\infty})$ acts transitively from the left on the set $\St(V,V\perp H^{\infty})$ via $(f,g)\mapsto f\circ g$.
The stabilizer of the element $i_V:V \to V\perp H^{\infty}:v\mapsto (v,0)$ of $\St(V,V\perp H^{\infty})$ is the subgroup $O(H^{\infty}) \subset O(V\perp H^{\infty}): A \mapsto 1_V\perp A$.
Therefore, we obtain an isomorphism of presheaves of sets
\begin{equation}
\label{OmodOisSt}
O(H^{\infty})\backslash O(V \perp H^{\infty}) \cong \St(V,V\perp H^{\infty}): f \mapsto f\circ i_V.
\end{equation}

\begin{proposition}
\label{prop:StContract}
Let $V$ be an inner product space over a commutative ring $R$ with $\frac{1}{2}\in R$.
Then the simplicial set
$$\St(V,V\perp H^{\infty}_{\Delta R})$$
is a contractible Kan set.
\end{proposition}

\begin{proof}
Contractibility follows from Proposition \ref{appx:XmodG=YmodG} applied to the $O(H^{\infty}_{\Delta^{\bullet}R})$ equivariant homotopy equivalence $O(H^{\infty}_{\Delta^{\bullet}R}) \subset O(V\perp H^{\infty}_{\Delta^{\bullet}R})$ of Lemma \ref{lem:OH=OVH}
together with the isomorphism (\ref{OmodOisSt}).
The simplicial set is fibrant, by Proposition \ref{appx:GmodHKan}.
\end{proof}

The group $O(V)$ of isometries of $V$ acts from the right on $\St(V,U)$ via $(f,g)\mapsto fg$ for $f\in \St(V,U)$ and $g\in O(V)$.
The map $\St(V,U) \to \GrO_V(U): f \mapsto \im(f)$ factors through the quotient map $\St(V,U) \to \St(V,U)/O(V)$ and yields an isomorphism of (presheaves of) sets
\begin{equation}
\label{eqn:St/O=GrO}
\St(V,U)/O(V) \cong \GrO_V(U): f \mapsto \im(f).
\end{equation}

For an inner product space $V$ over $R$ and a symmetric bilinear form $R$-module $U$, let $\calE_V(U)$ be the category whose objects are the $R$-module maps $V\to U$ respecting forms and where a map from $a:V \to U$ to $b:V \to U$ is a map $c:\im(a) \to \im(b)$ of inner product spaces such that the diagram
$$\xymatrix{
V \ar[r]^a \ar[rd]_b & \im(a) \ar[d]^c\\
& \im(b)
}
$$
commutes.
Note that the set of objects of $\calE_V(U)$ is the set $\St(V,U)$.
The group $O(V)$ acts freely from the right on $\calE_V(U)$ via 
$$\calE_V(U)\times O(V) \to \calE_V(U): (a,g)\mapsto ag,$$
the inclusion of zero simplices $\St(V,U) \to \calE_V(U)$ is $O(V)$-equivariant, and the functor $\calE_V(U) \to \IPS_V(U):a \mapsto \im(a)$ induces an isomorphism of simplicial sets
$$\calE_V(U)/O(V) \cong \IPS_V(U).$$

\begin{lemma}
\label{lem:EVcontr}
The category $\calE_V(V\perp H^{\infty})$ is contractible.
\end{lemma}

\begin{proof}
The category $\calE_V(V\perp H^{\infty})$ is non-empty as it has $V \to V\perp H^{\infty}: v\mapsto (v,0)$ as object.
Every object in $\calE_V(V\perp H^{\infty})$ is an initial object.
Hence, this category is contractible.
\end{proof}

\begin{proof}[Proof of Proposition \ref{prop:GrO0=KO0}]
The map of simplicial sets
$$\St(V,V\perp H^{\infty})(\Delta R) \to \calE_V(V\perp H^{\infty})(\Delta R)$$
is $O(V_{\Delta R})$-equivariant, the simplicial group $O(V_{\Delta R})$ acts freely on both sides, and the map is a non-equivariant weak equivalence (of contractible simplicial sets), by Proposition \ref{prop:StContract} and Lemma \ref{lem:EVcontr}.
By Lemma \ref{appx:XmodG=YmodG}, the map on quotient simplicial sets
$\GrO_V(V\perp H^{\infty}_{\Delta R}) \to \IPS_V(V\perp H^{\infty}_{\Delta R})$
is also a weak equivalence.
Finally, the inclusion $BO(V) \subset B\IPS_V(V\perp H^{\infty})$ is a weak equivalence since $\IPS_V(V\perp H^{\infty})$ is a connected groupoid.
\end{proof}

\section{$E_{\infty}$-spaces and the end of the proof of Theorem \ref{thm:simplHtpyEq}}
\label{sec:EinftyEndofPf}

Even though diagram (\ref{eqn:GrOKOdiagram}) is a map of homotopy fibrations which is a weak equivalence of simplicial sets on base and fibres, we can't conclude yet that the map on total spaces is a weak equivalence as well.
This will be done by establishing that all maps in diagram (\ref{eqn:GrOKOdiagram}) are maps of $E_{\infty}$-spaces; see Proposition \ref{prop:GrOEinf}.
%The homotopy equivalence $\IPS_{\bullet}(\Delta R) \sim \widetilde{GW}(\Delta R)$ is in Proposition \ref{}.

Informally, the ``linear isometries'' operad $\E$ has $n$-th space the space of isometric embeddings of $(H^{\infty})^n$ into $H^{\infty}$ with $\Sigma_n$ permuting the $n$ factors in $(H^{\infty})^n$.
More precisely,
for a commutative ring $R$, let $\E(n)(R)$ be the set
$$\E(n)(R)=\lim_{V \subset H_R^{\infty}}\St(V^n,H_{R}^{\infty}).$$
The inverse limit ranges over (a cofinal subset of) the category $\calH$ of all finitely generated non degenerate $V \subset H^{\infty}_R$, and $V^n = V \perp ... \perp V$ denotes $n$-fold orthogonal sum.
The permutation group $\Sigma_n$ on $n$-letters acts on $\E(n)$ by permuting the factors of $V^n$.
This action is free.
By Proposition \ref{prop:StContract}, the simplicial sets $\St(V^n,H_{\Delta R}^{\infty})$ are contractible Kan sets.
By Proposition \ref{appx:GmodHKan}, for $W \subset V$, the transition maps
$\St(V^n,H_{\Delta R}^{\infty}) \to \St(W^n,H_{\Delta R}^{\infty})$ are Kan fibrations in view of the identification (\ref{OmodOisSt}).
It follows that 
$$\E(n)(\Delta R) = \lim_k\St(H^k\perp ... \perp H^k,H^{\infty})(\Delta R)$$
is a contractible Kan set with a free $\Sigma_n$-action.
We define the structure maps of the operad $\E$ by
$$\E(k)\times \E(j_1) \times \cdots\times \E(j_k) \to \E(j_1 + \cdots +j_k) : f, g_1,...,g_k \mapsto f\circ (g_1 \perp ... \perp g_k)$$
Thus, we have proved the following lemma.

\begin{lemma}
\label{lem:EisEinfty}
Let $R$ be a commutative ring with $\frac{1}{2}\in R$.
Then the operad $\E(\Delta R)$ defined above is an $E_{\infty}$-operad.
\Qed
\end{lemma}

\begin{proposition}
\label{prop:GrOEinf}
For any commutative ring $R$ with $\frac{1}{2}\in R$, the map 
$$\GrO_{\bullet}(\Delta R) \to \IPS_{\bullet}(\Delta R)$$
 is a map of group complete $E_{\infty}$-spaces.
\end{proposition}

\begin{proof}
We make $\IPS_{\bullet}(R)$ into a module over the operad $\E$.
The inclusion of zero-simplices $\GrO_{\bullet}(R) \to \IPS_{\bullet}(R)$ will respect this action.
So, the proposition will follow from Lemma \ref{lem:EisEinfty}.

To define the action of the operad $\E$, write $\IPS_{\bullet}$ as 
$$\IPS_{\bullet} = \colim_{V \subset H^{\infty}}\IPS_{|V|}(V^-\perp V^+)$$
where $V^-$ and $V^+$ are two copies of $V$ and for $V \subset W$ the transition map is defined by 
$$\IPS_{|V|}(V^-\perp V^+) \to \IPS_{|W|}(W^-\perp W^+): E \mapsto (W-V)^- \perp E,\ g\mapsto 1_{(W-V)^-}\perp g.$$
Now, the action of $\E$ on $\IPS_{\bullet}$ is defined by
$$\St(V_1\perp ... \perp V_k,W) \times \IPS_{|V_1|}(V_1^-\perp V_1^+) \times \cdots \times \IPS_{|V_k|}(V_k^-\perp V_k^+) \longrightarrow \IPS_{|W|}(W^-\perp W^+)$$
where for $g\in \St(V_1\perp ... \perp V_k,W)$, the functor
$$\IPS_{|V_1|}(V_1^-\perp V_1^+) \times \cdots \times \IPS_{|V_k|}(V_k^-\perp V_k^+) \longrightarrow \IPS_{|W|}(W^-\perp W^+)$$
sends the object
$(E_1,...,E_k)$ to 
$$(W-g(V_1 \perp ... \perp V_k))^- \perp g(E_1 \perp ... \perp E_k)$$
and the map $(e_1,...,e_k):(E_1,...,E_k) \to (E'_1,...,E'_k)$ to 
$$1_{(W-g(V_1 \perp ... \perp V_k))^-}\perp g_{|E'_1}\circ  e_1 \circ g^{-1}_{|E_1}\perp \cdots \perp
g_{|E'_k}\circ e_k\circ  g^{-1}_{|E_k}.$$
\end{proof}

\begin{corollary}
\label{cor:GrOisS}
Let $R$ be a connected regular ring with $\frac{1}{2}\in R$.
Then the map 
$$\GrO_{\bullet}(\Delta R) \to \IPS_{\bullet}(\Delta R)$$
is a weak equivalence of simplicial sets.
\end{corollary}

\begin{proof}
This follows from the map of homotopy fibrations (\ref{eqn:GrOKOdiagram})
in view of Propositions \ref{prop:GrO0=KO0} and \ref{prop:GrOEinf}.
\end{proof}

\begin{proposition}
\label{prop:SGWHweak}
Let $R$ be a connected regular noetherian ring with $\frac{1}{2}\in R$.
Then the map (\ref{eqn:CatSToGW}) induces a weak equivalence of simplicial sets 
$$\catS_{\bullet}(\Delta R) \stackrel{\sim}{\rightarrow} \widetilde{\GW}(\Delta R).$$
\end{proposition}

\begin{proof}
By the Group Completion Theorem \cite[Theorem, p. 221]{quillenGrayson}, the map
$$\catS_{\bullet}(R) \rightarrow \widetilde{\GW}(R)$$
induces an isomorphism on integral homology groups.
It follows that the map in the proposition is an isomorphism on integral homology groups as well.
It is well-known that $GW(\Delta R) \simeq GW(R)$ and hence $\widetilde{\GW}(\Delta R)\simeq \widetilde{\GW}(R)$ are group complete $H$-spaces.
By Proposition \ref{prop:GrOEinf}, the same is true for $\catS_{\bullet}(\Delta R)$.
Therefore, the map in the proposition is indeed a weak equivalence of simplicial sets.
\end{proof}

\section{Geometric models for $GW^n$}

In this section we will prove Theorem \ref{thm:8spaces}.

\begin{proposition}
\label{prop:BetGrOdot}
Let $R$ be a regular noetherian ring with $\frac{1}{2}\in R$.
Then the map (\ref{eqn:GrOIPSBetOstable}) induces a weak equivalence of simplicial sets
$$\GrO_{\bullet}(\Delta R) \stackrel{\sim}{\longrightarrow} (B_{et}O)(\Delta R).$$
In particular,
for any regular noetherian scheme $S$ with $\frac{1}{2}\in \Gamma(S,O_S)$,
the canonical map $\GrO_{\bullet} \to B_{et}O$ is isomorphism in $\H_{\bullet}(S)$.
\end{proposition}

\begin{proof}
For connected $R$, this follows from Corollary \ref{cor:GrOisS} in view of
Lemma \ref{lem:IPSisBetO}.
Since both sides convert finite disjoint unions into cartesian products, we are done. 
\end{proof}

Write $\underline{\Z}$ for the constant sheaf associated with the constant presheaf $\Z$.
Recall that the presheaf $\pi_0B_{et}O$ is homotopy invariant on regular noetherian rings $R$ with $\frac{1}{2}\in R$ since on affine schemes it is the kernel of the rank map $GW \to \underline{\Z}$.
Similarly, the presheaves $\pi_0B_{et}GL$ and $\pi_0B_{et}Sp$ are also homotopy invariant on affine schemes.

Note that in the next theorem, we have $B_{et}G=B_{Nis}G=B_{Zar}G$ for $G=GL$ and $Sp$ but not for $O$.

\begin{theorem}
\label{ZBetGsiKG}
The canonical maps of presheaves of simplicial sets
$$
\begin{array}{llllll}
\underline{\Z} \times  B_{et}O & \to &  GW, & \hspace{4ex} O & \to & \Omega_{S^1}GW,\\
\underline{\Z} \times  B_{et}GL & \to & K, & \hspace{4ex} GL & \to & \Omega_{S^1}K,\\
\underline{\Z} \times  B_{et}Sp & \to &  GW^2, & \hspace{4ex} Sp & \to & \Omega_{S^1}GW^2
\end{array}
$$
are weak equivalences of simplicial sets when evaluated at $\Delta R$ for any  regular noetherian ring $R$ (with $\frac{1}{2}\in R$ in case of $O$ and $Sp$).
In particular, all these maps are $\A^1$-weak equivalences.
\end{theorem}

\begin{proof}
The first statement for the orthogonal group was proved for connected rings in Theorem \ref{thm:simplHtpyEq}.
Source and target of the map convert finite disjoint unions into cartesian products.
So, the case of non-connected rings follows.
For the second statement, consider the sequence
$$BO \to B_{et}O \to \pi_0B_{et}O$$
which is section-wise a homotopy fibration.
Since the base of the fibration is homotopy invariant on affine schemes,
the sequence of simplicial sets
$$(BO)(\Delta R) \to (B_{et})(\Delta R) \to (\pi_0B_{et}O))(\Delta R)$$
is a homotopy fibration with discrete base; see Proposition \ref{prop:SimplFib}.
It follows that the spaces $(BO)(\Delta R)= B(O(\Delta R)) $,  $(B_{et}O)(\Delta R)$ and $(\underline{\Z} \times B_{et}O)(\Delta R) \simeq GW(\Delta R)$ all have equivalent $S^1$-loop spaces.
But $\Omega_{S^1} B(O(\Delta R)) \simeq O(\Delta R)$ as is the case for any simplicial group in place of $O(\Delta R)$.

The case of the symplectic groups is {\em mutatis mutandis} the same as the orthogonal case replacing symmetric forms with alternating forms through-out.

The case of the general linear group is also {\em mutatis mutandis} the same provided one uses the correct dictionary.
``Inner product spaces'' should be replaced by ``finitely generated projective modules''.
``Maps respecting forms'' $(V,\ffi) \to (V',\ffi')$ are replaced by {\em direct maps} $(i,q):P \to P'$, that is, pairs of maps $i:  P \to P'$, $q:P' \to P$ such that $qi = 1_P$.
Composition of direct maps are composition of the $i$'s and $q$'s.
A {\em direct submodule} of a projective module $Q$ therefore is a submodule $i:P \subset Q$ together with a retract $q:Q \to P$ such that $qi=1$.
The {\em direct complement} of a direct submodule $(i,q):P \subset Q$ is the direct submodule $Q-P = \im(1_Q-iq)\subset Q$ equipped with the retraction $q-iq:Q \to (Q-P)$. 
Note that $P\oplus (Q-P) = Q$ (as submodules of $Q$).
The index category $\calH = \{ V \subset H^{\infty}\}$ in the definition of $\GrO_{\bullet}$ and $\IPS_{\bullet}$ gets replaced by the category $\calH'$ of finitely generated direct submodules of $R^{\infty}=\bigoplus_{\N}R$.
Direct inclusions, that is inclusions together with retracts, make $\calH'$ into a filtered category.
With these definitions, the details of the proof for $GL$ are left as an exercise.
\end{proof}

The following lemma applies to groups such as $GL$, $O$, $Sp$ and the various forgetful and hyperbolic maps between them.
Note that $(B_{et}G)(\Delta R)$ is an $E_{\infty}$-space for $G=GL$, $O$, $Sp$, by Theorem \ref{ZBetGsiKG}, or Propositions \ref{prop:GrOEinf} and \ref{prop:BetGrOdot} and their analogs for $Sp$ and $GL$.

\begin{lemma}
\label{lem:GmodHBet}
Let $G$ be a presheaf of groups on $\Sm_S$, and let $H\leq G$ be a presheaf of subgroups.
Assume that for $\Spec R \in \Sm_S$ the map 
$(B_{et}H)(\Delta R) \to  (B_{et}G)(\Delta R)$ is a map of group complete $E_{\infty}$ spaces. 
Assume further that the presheaves $\pi_0B_{et}G$ and $\pi_0B_{et}H$ are homotopy invariant on affines.
Then the canonical sequence 
$$(G/H)_{et} \to  B_{et}H \to  B_{et}G$$
is a homotopy fibration of simplicial sets when evaluated
at $\Delta R$ for any affine $\Spec R \in \Sm_S$.
\end{lemma}

\begin{proof}
Write $\tilde{B}H$ for $(EG)/H$ and recall that the map $BH = (EH)/H \to (EG)/H = \tilde{B}H$ is a weak equivalence on all sections; see Proposition \ref{appx:XmodG=YmodG}.
The sequence of presheaves
$G/H \to \tilde{B}H \to BG$ is a fibration sequence of simplicial sets 
on all sections; see Proposition \ref{appx:GmodHKan}.
Taking fibrant replacements in the etale topology (or any other topology, say, with enough points) preserves section-wise homotopy fibrations.
Therefore, the sequence
$(G/H)_{et} \to \tilde{B}_{et}H \to B_{et}G$
is a homotopy fibration on all sections.
Consider the commutative diagram of presheaves
$$\xymatrix{
G/H \ar[r] \ar[d] & \tilde{B}H \ar[r]\ar[d] & BG \ar[d]\\
(G/H)_{et} \ar[r] \ar[d] & \tilde{B}_{et}H \ar[r]\ar[d] & B_{et}G \ar[d]\\
X \ar[r] & \pi_0\tilde{B}_{et}H \ar[r] & \pi_0B_{et}G
}$$
where $X$ is the homotopy fibre (in this case, the kernel) of $\pi_0\tilde{B}_{et}H \to \pi_0B_{et}G$.
In this diagram, all rows and columns are homotopy fibrations, and the bottom row is homotopy invariant on affines.
Moreover, the lower vertical maps are surjective on $\pi_0$ (the left one because of the long exact sequence of homotopy groups associated with the middle row).
For $\Spec R \in \Sm_S$, we therefore obtain a commutative diagram of simplicial sets
$$\xymatrix{
(G/H)(\Delta R) \ar[r] \ar[d] & (\tilde{B}H)(\Delta R) \ar[r]\ar[d] & (BG)(\Delta R) \ar[d]\\
(G/H)_{et}(\Delta R) \ar[r] \ar[d] & (\tilde{B}_{et}H)(\Delta R) \ar[r]\ar[d] & (B_{et}G)(\Delta R) \ar[d]\\
X(\Delta R) \ar[r] & (\pi_0\tilde{B}_{et}H)(\Delta R) \ar[r] & (\pi_0B_{et}G)(\Delta R)
}$$
in which the columns are homotopy fibrations, by Proposition \ref{prop:SimplFib}, 
the bottom row is a homotopy fibration since it is the same as the bottom row of the previous diagram, and the top row is a homotopy fibration, by Proposition \ref{appx:GmodHKan}, since $(BN)(\Delta R) = B(N(\Delta R))$ for any presheaf of groups $N$.
Furthermore, the lower vertical maps are surjective on $\pi_0$ since this was also the case in the previous diagram.
The left column homotopy fibration maps to the homotopy fibration obtained by taking the homotopy fibres of the right horizontal maps.
By the five lemma applied to the long exact sequence of homotopy groups (in which all homotopy groups and sets are abelian groups as all spaces involved are group complete $E_{\infty}$-spaces, and the last non-trivial maps are surjective) these two homotopy fibrations are weakly equivalent.
It follows that the middle row is also a homotopy fibration.
Since $BH \to \tilde{B}H$ is a section wise weak equivalence, the same is true for $B_{et}H \to \tilde{B}_{et}H$ and $\Sing B_{et}H \to \Sing\tilde{B}_{et}H$.
This proves the claim.
\end{proof}

\begin{theorem}
\label{thm:Modles}
There are canonical maps of simplicial presheaves
$$\begin{array}{llllll}
(Sp/GL)_{et} & \to & GW^1, & \hspace{4ex} (GL/O)_{et} & \to & \Omega_{S^1}GW^1,\\
(O/GL)_{et} & \to & GW^3, & \hspace{4ex} (GL/Sp)_{et} & \to & \Omega_{S^1}GW^3
\end{array}
$$
which are weak equivalences of simplicial sets when evaluated at $\Delta R$ for $R$ a regular noetherian ring with $\frac{1}{2}\in R$.
In particular, all these maps are $\A^1$-weak equivalences.
\end{theorem}

\begin{proof}
For $n\in \Z$ there are homotopy fibrations 
$GW^n \stackrel{F}{\to} K \stackrel{H}{\to} GW^{n+1}$
where $F$ and $H$ denote forgetful and hyperbolic functor, respectively \cite[Theorem 6.1]{myGWDG}.
Since $GW^n$ and $K$ are homotopy invariant on regular rings \cite[Theorem 9.8]{myGWDG}, we have homotopy fibrations
$$GW^n(\Delta R) \stackrel{F}{\to} K(\Delta R) \stackrel{H}{\to} GW^{n+1}(\Delta R)$$
for any regular noetherian $R$ with $\frac{1}{2}\in R$.
The results now follow from Theorem \ref{ZBetGsiKG} and Lemma \ref{lem:GmodHBet}.
\end{proof}

\begin{remark}
\label{CounterexToMorelVoev}
The proof given in \cite{MorelVoevodsky} that $\Z \times Gr_{\bullet} \cong \Z\times BGL \cong K$ in $\H_{\bullet}(S)$ formally rests on \cite[Proposition 1.9, p.126]{MorelVoevodsky}. This proposition, however, is false as the following example shows.

Let $T$ be the one-point-site, so that $\H_s(T)$ is the homotopy category of simplicial sets.
Let $R$ be a non-zero ring and $M=\bigsqcup_{n\in \N}BGL_n(R)$ be the monoid defined by $BGL_m \times BGL_n \to BGL_{m+n}: (A,B) \mapsto \left(\begin{smallmatrix}A&0\\ 0&B\end{smallmatrix}\right)$.
This monoid is commutative in $\H_s(T)$ 
because it is the classifying space of the symmetric monoidal category of finite rank free $R$-modules with isomorphisms as morphisms.
Alternatively, the monoid multiplication is commutative because
$\left(\begin{smallmatrix}A&0\\ 0&B\end{smallmatrix}\right) = 
\left(\begin{smallmatrix}0&1\\ 1&0\end{smallmatrix}\right) \left(\begin{smallmatrix}B&0\\ 0&A\end{smallmatrix}\right)\left(\begin{smallmatrix}0&1\\ 1&0\end{smallmatrix}\right)^{-1}$ and conjugation $c_g:G \to G: h\mapsto ghg^{-1}$  induces a map $c_g:BG \to BG$ on classifying spaces which is homotopic to the identity map.
%The easiest way to see the last (and well-known) statement is to consider $G$ as a category with one object and $G$ as homomorphism set. Then $c_g:G \to G$ is a functor and $g:id \to c_g$ defines a natural transfrmation which induces the homotopy between $id = B(id)$ and $B(c_g)$ on classifying spaces.
If we believe the conclusion of \cite[Proposition 1.9, p.126]{MorelVoevodsky}, then we would have a weak equivalence of simplicial sets $Z \times BGL(R) \sim \Omega B(M)$ which cannot exist since $\pi_1$ of the left hand side is non-abelian whereas $\pi_1$ of the right hand side is abelian.
%The incorrect statement in \cite[Proposition 1.9, p.126]{MorelVoevodsky}, however, doesn't invalidate the $\A^1$-weak equivalences above since the natural map 
%$\Z \times BGL(\Delta R) \to K(\Delta R)$ is a homology isomorphism of group complete $H$-spaces for every local ring $R$, hence it is a weak equivalence of simplicial sets, hence $\Z \times BGL \to K$ is an $\A^1$-weak equivalence.
%Furthermore, the $\A^1$-weak equivalence $BGL \sim \Gr_{\bullet}$ follows from \cite[Proposition 3.7]{MorelVoevodsky}.
\end{remark}

\appendix

\section{Simplicial sets}

We collect a few well-known facts about simplicial sets which are used throughout the paper. 
The standard reference nowadays is \cite{GoerssJardine}.

\begin{lemma}
\label{lem:XYZfibrations}
Given a sequence $X \to Y \to Z$ of simplicial sets in which $X \to Y$ is a surjective fibration and the composition $X \to Z$ is a fibration.
Then the map $Y \to Z$ is a fibration.
\end{lemma}

\begin{proof}
One checks that $Y \to Z$ has the right lifting property with respect to the maps $\Lambda^k_n \subset \Delta_n$.
Any map $\Lambda^k_n \to Y$ lifts to a map $\Lambda_n^k \to X$.
This is because we can lift the image of a zero simplex in $\Lambda^k_n$ (as $X \to Y$ is surjective) and extend this lift to all of $\Lambda^k_n$ since $X \to Y$ is a fibration and the inclusion of a point into $\Lambda^k_n$ is an acyclic cofibration. 
Then the map $\Delta_n \to Z$ lifts to $X$ since $X \to Z$ is a fibration.
Composing this lift with the map $X \to Y$ yields the required map.
\end{proof}

\begin{lemma}
\label{lem:cartIffFibration}
Given a cartesian square of simplicial sets
$$\xymatrix{X \ar[r] \ar[d] & Y \ar[d] \\
Z \ar[r] & W}$$
in which the right (and hence the left) vertical map is a surjective fibration.
Then the upper horizontal map is a weak equivalence if and only if the lower horizontal map is a weak equivalence.
\end{lemma}

\begin{proof}
By properness of the model category of simplicial sets, if $Z \to W$ is a weak equivalence then so is $X \to Y$.

Assume that $X \to Y$ is a weak equivalence.
Factoring $Z \to W$ into an acyclic cofibration and a fibration and pulling $Y \to W$ along that fibration, we can reduce to showing the claim in case $Z \to W$ is a fibration (and $X\to Y$ an acyclic fibration).
Then we need to show that $Z \to W$ has the right lifting property with respect to all inclusions $\partial \Delta_n \subset \Delta_n$.

Given a map from  $\partial \Delta_n \subset \Delta_n$ to $Z \to W$.
Choose a lift of $\Delta_n \to W$ to $Y$ which exists since $Y \to W$ is surjective. 
The universal property of $Y$ as a pull-back yields a lift of $\partial \Delta_n \to Z$ to $X$ making all diagrams commute.
Since $X \to Y$ is an acyclic fibration, the map $\Delta_n \to Y$ lifts to $X$.
Composing this map with $X \to Z$ yields the required lift.
\end{proof}

\begin{lemma}\cite[Lemma I.3.4]{GoerssJardine}
Let $G$ be a simplicial group.
Then $G$ is fibrant as simplicial set.
\end{lemma}

\begin{proposition}
\label{prop:XtoX/GKan}
Let $G$ be a simplicial group acting freely from the right on a simplicial set $X$. 
Then the quotient map $X \to X/G$ is a Kan fibration.
\end{proposition}

\begin{proof}
We need to show that the map $X \to X/G$ has the right lifting property with respect to the standard generating acyclic cofibrations $\Lambda^n_k \subset \Delta^n$.
Take a commutative square for which we have to find a lift.
Pulling back along the map $\Delta^n \to E/G$, we see that it suffices to show the claim of the proposition in case $X/G = \Delta^n$.
In this case, choose a section $s:\Delta^n \to X$ and define the map $\Delta^n\times G \to X: (x,g)\mapsto s(x)g$.
Since $G$ acts freely on $X$, this map is an isomorphism.
Finally, the projection map $\Delta^n\times G \cong X \to \Delta^n$ is a fibration because $G$ is fibrant.
\end{proof}

\begin{proposition}
\label{appx:GmodHKan}
Let $G$ be a simplicial group and let $H\leq G$ be a simplicial subgroup.
Let $X$ be a simplicial set with a free $G$-action from the right.
Then the map $X/H \to X/G$ is a Kan fibration.
In particular, the map $G \to G/H$ is a Kan fibration, and the simplicial set 
$G/H$ is fibrant.
\end{proposition}

\begin{proof}
We apply Lemma \ref{lem:XYZfibrations} to the sequence
$X \to X/H \to X/G$ using Proposition \ref{prop:XtoX/GKan}.
So, $X/H \to X/G$ is a Kan fibration.
Applied to $X=G$ and the inclusion of subgroups $\{e\} \subset H$, we obtain
the Kan fibration $G\to G/H$.
Applied to $X=G$ and the inclusion of groups $H \subset G$, we obtain the Kan fibration $G/H \to G/G=*$.
\end{proof}

\begin{proposition}
\label{appx:XmodG=YmodG}
Let $G$ be a simplicial group acting freely on the right on the simplicial sets $X$ and $Y$.
Let $X \to Y$ be a $G$-equivariant map which is a non-equivariant weak equivalence (that is, a weak equivalence forgetting the action). 
Then $X/G \to Y/G$ is a weak equivalence.
\end{proposition}

\begin{proof}
Apply Lemma \ref{lem:cartIffFibration} with $Z \to W$ the map $X/G \to Y/G$ and vertical maps the quotient maps.
The diagram is cartesian because $G$ acts freely on $X$ and $Y$, and the 
right vertical map is a surjective fibration, by Proposition \ref{prop:XtoX/GKan}.
\end{proof}

For a bisimplicial set $X$, denote by $\diag X$ the diagonal simplicial set $(\diag X)_n = X_{n,n}$.
The following proposition follows from the Bousfield-Friedlander theorem \cite[Theorem IV.4.9]{GoerssJardine} or from Mather's Cube Theorem \cite{Mather}.

\begin{proposition}
\label{prop:SimplFib}
Let $X_{\bullet \bullet} \to Y_{\bullet \bullet} \to Z_{\bullet \bullet}$ be a sequence of bisimplicial sets such that for all $p\in \N$, the sequence of simplicial sets
$X_{p \bullet} \to Y_{p \bullet} \to Z_{p \bullet}$ is a homotopy fibration, and $Z_{\bullet\bullet}$ is constant in the $p$-direction, that is, $Z_{p_1,\bullet} \to Z_{p_2,\bullet}$ is the identity for all simplicial operators $[p_2]\to [p_1]$.
Then the sequence of diagonal simplicial sets
$$\diag X \to \diag Y \to \diag Z$$
is a homotopy fibration.
\end{proposition}

In order to construct certain maps in the body of our paper we will have to use homotopy colimits.
The reason is that the $K$-theory and hermitian $K$-theory spaces are homotopy colimits themselves; see Remark \ref{rem:SSasHocolim}.
Below we recall the construction within the category of small categories, and in Lemma \ref{lem:FilteringHocolim=colim} we recall a well-known basic fact that we will need.

\begin{definition}[Homotopy colimits]
\label{dfn:hocolim}
Let $\C$ be a small category and 
$\F: \C \to \Cat$ a functor from $\C$ to the category $\Cat$ of small categories. 
The homotopy colimit 
$$\hocolim_{\C}\F$$
is the category whose objects are pairs
$(X,A)$ with $X$ and object of $\C$ and $A$ an object of $\F(X)$.
A map from $(X,A)$ to $(Y,B)$ is a pair $(x,a)$ where $x:X \to Y$ is a map in $\C$ and $a: \F(x)A \to B$ is a map in $\F(Y)$.
Composition $(y,b)\circ (x,a)$ of $(y,b):(Y,B) \to (Z,C)$ and $(x,a):(X,A) \to (Y,B)$ is the map $(y\circ x, b\circ F(y)a)$.
\end{definition}

By a result of Thomason \cite{Thomason:hocolim}, the nerve simplicial set $N_*\hocolim_{\C}\F$ is naturally homotopy equivalent to the Bousfield-Kan homotopy colimit of the diagram $N_*\F:\C \to \Delta^{op}\Sets$ of simplicial sets.
We won't need this fact, but we will need the following special case.
For that, recall that a poset $(\scP,\leq)$ is considered a category with objects the elements of the poset and a unique map from $P\in \scP$ to $Q \in \scP$ if $P\leq Q$.
The poset $(\scP,\leq)$ is filtering if for every $P,Q\in \scP$ there is a $R\in \scP$ with $P,Q\leq R$.

\begin{lemma}
\label{lem:FilteringHocolim=colim}
Let $(\scP,\leq)$ be a filtering poset and let $\F:\scP \to \Cat$ be a functor from $\scP$ into the category $\Cat$ of small categories.
Then the functor of categories
$$\phi: \hocolim_{\scP}\F \to \colim_{\scP}\F: (P,A) \mapsto [P,A]$$
is a homotopy equivalence of simplicial sets.
\end{lemma}

\begin{proof}
By Quillen's theorem A \cite{quillen:higherI}, it suffices to show that for every object $[P,A]$ of the category $\colim_{\scP}\F$, the comma category
$(\phi \downarrow [P,A])$ is contractible.
For $A \in \F(P)$ and $P\leq Q$ write $A_Q$ for the object $\F(P\leq Q)A$ in $\F(Q)$ which is the image of $A$ under the functor $\F(P\leq Q): \F(P) \to \F(Q)$.
Contractibility of the comma category now follows from the equivalence of categories 
$$\colim_{P\leq Q \in \scP}(id\downarrow (Q,A_Q)) \cong (\phi \downarrow [P,A])$$
where for $Q \leq R$, the functor $(id\downarrow (Q,A_Q) \to (id\downarrow (R,A_R)$ sends $t:(T,B) \to (Q,A_Q)$ to $c \circ t: (T,B) \to (R,A_R)$ with $c:(Q,A_Q) \to (R,A_R)$ the map given by $id:A_R = \F(Q\leq R)A_Q \to A_R$.
The left-hand category is a filtered colimit over categories with initial objects, hence a filtered colimit over contractible categories.
Therefore, the left-hand category is contractible, and so is the right-hand category.
\end{proof}

\begin{definition}
\label{dfn:naiveA1htpy}
Let $k$ be a commutative ring and $F,G$ be simplicial presheaves on smooth affine $k$-schemes.
An {\em elementary $\A^1$-homotopy} between two maps $h_0,h_1:F \to G$ of presheaves is a map of presheaves $h:\A^1\times F \to G$ such that $h_i = h\circ j_i$, $i=0,1$ where $j_i: \Spec(k) \to \A^1$ corresponds to the evaluation $k[t] \to k: t \mapsto i$. 
Elementary homotopy generates an equivalence relation called {\em naive $\A^1$-homotopy}.
The following is a well-known fundamental fact from $\A^1$-homotopy theory.
%\edit{base ring $k$}
\end{definition}

\begin{lemma}
\label{lem:naiveA1htpy}
If $h_0,h_1: F \to G$ are naively $\A^1$-homotopic then for every $k$-algebra $R$, the maps $h_0,h_1: F(\Delta R) \to G(\Delta R)$ are simplicially homotopic.
\end{lemma}

\begin{proof}
It suffices to prove the claim for elementary $\A^1$-homotopy.
Let $h:\A^1 \times F \to G$ be an elementary homotopy between $h_0$ and $h_1$.
The $1$-simplex $id \in \A^1(\Delta^1) \cong \A^1(\A^1)$ of the simplicial set $\A^1(\Delta )$ defines a 
map of simplicial sets $\Delta^1 \to \A^1(\Delta _k) \to \A^1(\Delta _R)$ which induces the required homotopy 
$H: \Delta^1 \times F(\Delta _R) \to \A^1(\Delta _R) \times F(\Delta _R) \stackrel{h}{\to} G(\Delta _R)$.
\end{proof}

%\bibliography{gwdg.bib}

\begin{thebibliography}{Sch10b}

\bibitem[AF12]{AsokFaselCohClass}
Aravind Asok and Jean Fasel.
\newblock A cohomological classification of vector bundles on smooth affine
  threefolds.
\newblock {\em arXiv:1204.0770}, 2012.

\bibitem[GJ99]{GoerssJardine}
Paul~G. Goerss and John~F. Jardine.
\newblock {\em Simplicial homotopy theory}, volume 174 of {\em Progress in
  Mathematics}.
\newblock Birkh\"auser Verlag, Basel, 1999.

\bibitem[Gra76]{quillenGrayson}
Daniel Grayson.
\newblock Higher algebraic {$K$}-theory. {II} (after {D}aniel {Q}uillen).
\newblock In {\em Algebraic $K$-theory (Proc. Conf., Northwestern Univ.,
  Evanston, Ill., 1976)}, pages 217--240. Lecture Notes in Math., Vol. 551.
  Springer, Berlin, 1976.

\bibitem[Hor05]{hornbostel:A1reps}
Jens Hornbostel.
\newblock {$A^1$}-representability of {H}ermitian {$K$}-theory and {W}itt
  groups.
\newblock {\em Topology}, 44(3):661--687, 2005.

\bibitem[Jar01]{JardineStacks}
J.~F. Jardine.
\newblock Stacks and the homotopy theory of simplicial sheaves.
\newblock {\em Homology Homotopy Appl.}, 3(2):361--384, 2001.
\newblock Equivariant stable homotopy theory and related areas (Stanford, CA,
  2000).

\bibitem[Kar73]{karoubi:battelle}
Max Karoubi.
\newblock P\'eriodicit\'e de la {$K$}-th\'eorie hermitienne.
\newblock In {\em Algebraic $K$-theory, III: Hermitian $K$-theory and geometric
  applications (Proc. Conf., Battelle Memorial Inst., Seattle, Wash., 1972)},
  pages 301--411. Lecture Notes in Math., Vol. 343. Springer, Berlin, 1973.

\bibitem[Mat76]{Mather}
Michael Mather.
\newblock Pull-backs in homotopy theory.
\newblock {\em Canad. J. Math.}, 28(2):225--263, 1976.

\bibitem[MV99]{MorelVoevodsky}
Fabien Morel and Vladimir Voevodsky.
\newblock {${\bf A}^1$}-homotopy theory of schemes.
\newblock {\em Inst. Hautes \'Etudes Sci. Publ. Math.}, (90):45--143 (2001),
  1999.

\bibitem[PW10]{PaninWalterBO}
Ivan Panin and Charles Walter.
\newblock On the motivic commutative ring spectrum ${BO}$.
\newblock {\em arXiv:1011.0650}, 2010.

\bibitem[Qui73]{quillen:higherI}
Daniel Quillen.
\newblock Higher algebraic {$K$}-theory. {I}.
\newblock In {\em Algebraic {$K$}-theory, {I}: {H}igher {$K$}-theories ({P}roc.
  {C}onf., {B}attelle {M}emorial {I}nst., {S}eattle, {W}ash., 1972)}, pages
  85--147. Lecture Notes in Math., Vol. 341. Springer, Berlin, 1973.

\bibitem[Rio10]{riou}
Jo{\"e}l Riou.
\newblock Algebraic {$K$}-theory, {${\bf A}^1$}-homotopy and {R}iemann-{R}och
  theorems.
\newblock {\em J. Topol.}, 3(2):229--264, 2010.

\bibitem[Sch04]{mygiffen}
Marco Schlichting.
\newblock Hermitian {$K$}-theory on a theorem of {G}iffen.
\newblock {\em $K$-Theory}, 32(3):253--267, 2004.

\bibitem[Sch10a]{myEx}
Marco Schlichting.
\newblock Hermitian {$K$}-theory of exact categories.
\newblock {\em J. K-Theory}, 5(1):105--165, 2010.

\bibitem[Sch10b]{myMV}
Marco Schlichting.
\newblock The {M}ayer-{V}ietoris principle for {G}rothendieck-{W}itt groups of
  schemes.
\newblock {\em Invent. Math.}, 179(2):349--433, 2010.

\bibitem[Sch12]{myGWDG}
Marco Schlichting.
\newblock Hermitian ${K}$-theory, derived equivalences, and {K}aroubi's
  {F}undamental {T}heorem.
\newblock {\em arXiv:1209.0848}, 2012.

\bibitem[Tho79]{Thomason:hocolim}
R.~W. Thomason.
\newblock Homotopy colimits in the category of small categories.
\newblock {\em Math. Proc. Cambridge Philos. Soc.}, 85(1):91--109, 1979.

\bibitem[Zib11]{zibrowius}
Marcus Zibrowius.
\newblock Witt groups of complex cellular varieties.
\newblock {\em Doc. Math.}, 16:465--511, 2011.

\end{thebibliography}

\end{document}